\documentclass{amsart}

\usepackage[dvips]{graphicx}

\usepackage{charter}
\usepackage[T1]{fontenc}
\usepackage[a4paper,top=46mm,bottom=46mm,left=37mm,right=37mm]{geometry}

\usepackage{latexsym,amsfonts} 
\usepackage{amsmath,amssymb,graphics,setspace}
\usepackage{amsthm}
\usepackage{mathrsfs} 
\usepackage[T1]{fontenc}%

\usepackage[ruled,vlined,linesnumbered]{algorithm2e}

\usepackage{marvosym}

\usepackage{pstricks,pst-text,pst-grad,pst-node,pst-3dplot,pstricks-add,pst-poly,pst-coil} 
\usepackage{pst-fun,pst-blur} 

\usepackage{picins,graphicx}
\usepackage{paralist}
\usepackage{color}

\usepackage[verbose]{wrapfig}

\usepackage{subfigure}
\renewcommand{\thesubfigure}{\thefigure.\arabic{subfigure}}
\makeatletter
\renewcommand{\p@subfigure}{}
\renewcommand{\@thesubfigure}{\thesubfigure:\hskip\subfiglabelskip}
\makeatother

\usepackage{hyperref}
\hypersetup{linktocpage=true,colorlinks=true,linkcolor=blue,citecolor=blue,pdfstartview={XYZ 1000 1000 1}}

\usepackage{floatflt,graphicx}

\usepackage{graphicx}
\makeatletter
\DeclareFontFamily{U}{tipa}{}
\DeclareFontShape{U}{tipa}{bx}{n}{<->tipabx10}{}
\newcommand{\arc@char}{{\usefont{U}{tipa}{bx}{n}\symbol{62}}}%

\newcommand{\arc}[1]{\mathpalette\arc@arc{#1}}

\newcommand{\arc@arc}[2]{%
  \sbox0{$\m@th#1#2$}%
  \vbox{
    \hbox{\resizebox{\wd0}{\height}{\arc@char}}
    \nointerlineskip
    \box0
  }%
}
\makeatother

\makeatletter
\newcommand{\doublewedge}{\big@doubleop{\wedge}}
\newcommand{\big@doubleop}[1]{%
  \DOTSB\mathop{\mathpalette\big@doubleop@aux{#1}}\slimits@
}

\newcommand\big@doubleop@aux[2]{%
  \sbox\z@{$\m@th#1#2$}%
  \makebox[1.35\wd\z@][s]{$\m@th#1#2\hss#2$}%
}

\newcommand{\abs}[1]{\left|#1\right|}     
\newcommand{\norm}[1]{\left\|#1\right\|}  

\newcommand{\cl}{\mbox{cl}}
\newcommand{\Int}{\mbox{int}}
\newcommand{\bdy}{\mbox{bdy}}

\newcommand{\Nrv}{\mbox{Nrv}}
\newcommand{\near}{\delta} 
\newcommand{\dnear}{\delta_{\Phi}} 
\newcommand{\dcap}{\mathop{\cap}\limits_{\Phi}} 
\newcommand{\Dcap}{\mathop{\bigcap}\limits_{\Phi}} 

\newcommand{\sig}{\mbox{sig}}

\newcommand{\sh}{\mbox{sh}}

\newcommand{\cyc}{\mbox{cyc}}

\newcommand{\dfar}{{\not\delta}_{\Phi}} 
\newcommand{\sn}{\mathop{\delta}\limits^{\doublewedge}} 
\newcommand{\dcup}{\mathop{\cup}\limits_{\Phi}} 

\newcommand{\notfar}{\mathop{\not{\delta}}\limits^{\doublewedge}} 

\newtheorem{example}{Example}
\newtheorem{remark}{Remark}
\newtheorem{definition}{Definition}
\newtheorem{lemma}{Lemma}
\newtheorem{theorem}{Theorem}
\newtheorem{proposition}{Proposition}
\newtheorem{corollary}{Corollary}
\newtheorem{conjecture}{Conjecture}

\usepackage{xcolor}
\definecolor{light}{gray}{0.80}

\begin{document}

\title[Proximal Planar Shape Signatures]{Proximal Planar Shape Signatures.\\   Homology Nerves and Descriptive Proximity}

\author[James F. Peters]{James F. Peters}
\address{
Computational Intelligence Laboratory,
University of Manitoba, WPG, MB, R3T 5V6, Canada and
Department of Mathematics, Faculty of Arts and Sciences, Ad\.{i}yaman University, 02040 Ad\.{i}yaman, Turkey}
\thanks{The research has been supported by the Natural Sciences \&
Engineering Research Council of Canada (NSERC) discovery grant 185986 
and Instituto Nazionale di Alta Matematica (INdAM) Francesco Severi, Gruppo Nazionale per le Strutture Algebriche, Geometriche e Loro Applicazioni grant 9 920160 000362, n.prot U 2016/000036.}


\date{}

\dedicatory{Dedicated to  J.H.C. Whitehead and Som Naimpally}

\begin{abstract}
This article introduces planar shape signatures derived from homology nerves, which are intersecting 1-cycles in a collection of homology groups endowed with a proximal relator (set of nearness relations) that includes a descriptive proximity.   A 1-cycle is a closed, connected path with a zero boundary in a simplicial complex covering a finite, bounded planar shape.   The \emph{signature of a shape} $\sh A$ (denoted by $\sig(\sh A)$) is a feature vector that describes $\sh A$.   A signature $\sig(\sh A)$ is derived from the geometry, homology nerves, Betti number, and descriptive CW topology on the shape $\sh A$.   
Several main results are given, namely, (a) every finite, bounded planar shape has a signature derived from the homology group on the shape, (b) a homology group equipped with a proximal relator defines a descriptive Leader uniform topology and (c) a description of a homology nerve and union of the descriptions of the 1-cycles in the nerve have same homotopy type.
\end{abstract}

\maketitle
\tableofcontents

\section{Introduction}
This paper introduces shape signatures restricted to the Euclidean plane.   A finite, bounded \emph{planar shape} $A$ (denoted by $\sh A$) is a finite region of the Euclidean plane bounded by a simple closed curve and with a nonempty interior~\cite{Peters2017arXiv1708-04147planarShapes}.

\setlength{\intextsep}{0pt}
\begin{wrapfigure}[11]{R}{0.35\textwidth}
\begin{minipage}{4.0 cm}
\centering
\begin{pspicture}
(-0.5,-1.5)(4,1.5)
\psline[linestyle=solid]%
(0,0)(1,1)(3,1)(4,0)(3,-1)(1,-1)(0,0)
\psline[linestyle=solid]%
(1,-1)(1,1)(1.8,0.2)(3,1)(3,-1)(1.8,0.2)(1,-1)
\psdots[dotstyle=o, linewidth=1.2pt,linecolor = black, fillcolor = yellow]%
(0,0)(1,-1)(1,1)(1.8,0.2)(3,1)(3,-1)(1.8,0.2)(4,0)
\psline[linecolor=blue,arrowsize=0pt 5]{->}(0,0)(1,1)
\psline[linecolor=blue,arrowsize=0pt 5]{->}(1,1)(3,1)
\psline[linecolor=blue,arrowsize=0pt 5]{->}(3,1)(3,-1)
\psline[linecolor=blue,arrowsize=0pt 5]{->}(3,-1)(1,-1)
\psline[linecolor=blue,arrowsize=0pt 5]{->}(1,-1)(0,0)
\rput(0.3,0.75){$\boldsymbol{e_1}$}\rput(2.0,1.2){$\boldsymbol{e_2}$}
\rput(3.2,0.0){$\boldsymbol{e_3}$}\rput(2.0,-1.2){$\boldsymbol{e_4}$}
\rput(0.3,-0.75){$\boldsymbol{e_5}$}\rput(-0.3,-0.0){$\boldsymbol{v_1}$}
\end{pspicture}
\caption[]{Path}
\label{fig:1-path}
\end{minipage}
\end{wrapfigure}  

After covering a shape with a simplicial complex, the signature of a shape is derived from the characteristics of the simple closed connected paths derived from connections between vertices in the covering.     A \emph{path} in a simplicial complex is a sequence of connected simplexes.    A \emph{closed path} is a connected path in which one can start at any vertex $v$ in the path and traverse the path to reach $v$.   A \emph{simple closed path} contains no self intersections (loops).  A pair of adjacent simplexes $\sigma_1, \sigma_2$ are \emph{connected}, provided $\sigma_1, \sigma_2$ have a common part~\cite[\S IV.1, p. 169]{Bredon1997homologyTheory}.


A path is oriented, provided the path can be traversed in either forward (positive) or reverse (negative) direction.    In other words, for any pair of adjacent edges in an oriented path, we can choose one of the edges and the direction to take in traversing the edges (\emph{cf.}, M. Berger and G. Gostiaux~\cite[\S 0.1.3]{BergerGostiaux1988orientation} and J.W. Ulrich ~\cite[\S 2, p. 364]{Ulrich1970SIAMJAMorientedGraph} on oriented graphs).

\begin{example} {\bf Sample Connected 1-simplexes in a Simple Closed Path}.\\
Let $e_1, e_2, e_3, e_4, e_5$ be a sequence of oriented path containing 1-simplexes (edges)
as shown in Fig.~\ref{fig:1-path}.     The ordering of  the 0-simplexes (vertices) is suggested by the directed edges.  For example, $e_1\rightarrow e_2\rightarrow e_3\rightarrow e_4\rightarrow e_5\rightarrow e_1$ defines a path.   This path is closed, since $e_5\rightarrow e_1$ at the end of a traversal of the edges, starting at $v_1$.   This closed path is simple, since it has no loops.
\qquad \textcolor{blue}{\Squaresteel}
\end{example}


A triangulated shape $A$ (also denoted by $\sh A$)  is connected, provided there is an edgewise simple closed path between each pair of vertices in $\sh A$.   Let $K$ be a simplicial complex covering shape $\sh A$.   A 1-chain is a formal sum of edges leading from one vertex to another vertex on $K$.    A 1-cycle is a 1-chain with an empty boundary. 
Also let $\sigma_i$ denote the $i$th edge in a path in $K$, $C_1(K)$ a set of cycles on edges on $K$ and let $C_0(K)$ be a set of cycles on vertices on $K$.     Let $\sigma$ be a simplex spanned by the vertices $v_0,\dots,v_n$ in $K$.  For $p\geq 1$, the homomorphic mapping $\partial_p: C_1(K)\longrightarrow C_{0}(K)$
is defined by
\[
\partial_1\sigma = \mathop{\sum}\limits_{i=0}^n (-1)^i\left[v_0,\dots,v_n\right] =  \mathop{\sum}\limits_{i=0}^n \sigma_i.
\]
The alternating signs on the terms indicate the simplexes are oriented, which means that for each positive term $+v_j$, there is a corresponding $-v_j, 0 \leq j\leq n$.   The signs are inserted to take path orientation into account, so that all faces of a simplex are coherently oriented~\cite[\S 2.1]{Hatcher2002CUPalgebraicTopology}. 


\setlength{\intextsep}{0pt}
\begin{wrapfigure}[10]{R}{0.43\textwidth}
\begin{minipage}{4.3 cm}
\centering
\begin{pspicture}
(-1.1,-1.5)(4,1.5)
\psline[linestyle=solid]%
(0,0)(1,1)(3,1)(4,0)(3,-1)(1,-1)(0,0)
\psline[linestyle=solid]%
(1,-1)(1,1)(1.8,0.2)(3,1)(3,-1)(1.8,0.2)(1,-1)
\pspolygon[fillstyle=solid,fillcolor=lightgray](3,1)(3,-1)(1.8,0.2)
\psdots[dotstyle=o, linewidth=1.2pt,linecolor = black, fillcolor = yellow]%
(0,0)(1,-1)(1,1)(1.8,0.2)(3,1)(3,-1)(1.8,0.2)(4,0)
\psline[linecolor=blue,arrowsize=0pt 5]{->}(0,0)(1,1)
\psline[linecolor=blue,arrowsize=0pt 5]{->}(1,1)(3,1)
\psline[linecolor=blue,arrowsize=0pt 5]{->}(3,1)(3,-1)
\psline[linecolor=blue,arrowsize=0pt 5]{->}(3,-1)(1,-1)
\psline[linecolor=blue,arrowsize=0pt 5]{->}(1,-1)(0,0)
\psline[linecolor=blue,arrowsize=0pt 5]{->}(2.98,-0.98)(1.8,0.2) 
\psline[linecolor=blue,arrowsize=0pt 5]{<-}(1.81,0.21)(2.98,0.98) 
\rput(0.3,0.75){$\boldsymbol{e_1}$}\rput(2.0,1.2){$\boldsymbol{e_2}$}\rput(2.2,-0.5){$\boldsymbol{e_6}$}
\rput(3.2,0.0){$\boldsymbol{e_3}$}\rput(2.0,-1.2){$\boldsymbol{e_4}$}\rput(2.2,0.7){$\boldsymbol{e_7}$}
\rput(0.3,-0.75){$\boldsymbol{e_5}$}\rput(2.5,0.1){\textcolor{black}{\large $\boldsymbol{H_o}$}}
\rput(3,1.25){\textcolor{black}{\large $\boldsymbol{v_3}$}}\rput(1,1.25){\textcolor{black}{\large $\boldsymbol{v_2}$}}
\rput(-0.25,0){\textcolor{black}{\large $\boldsymbol{v_1}$}}\rput(3,-1.25){\textcolor{black}{\large $\boldsymbol{v_4}$}}
\rput(1,-1.25){\textcolor{black}{\large $\boldsymbol{v_5}$}}\rput(1.4,0.2){\textcolor{black}{\large $\boldsymbol{v_6}$}}
\end{pspicture}
\caption{\bf 1-cycle}
\label{fig:chainMap}
\end{minipage}
\end{wrapfigure} 

The maps $\partial_n$ are called \emph{chain maps} (or \emph{simplicial boundary maps}).   Each chain map $\partial_n$ is a \emph{homomorphism}.   The sum of the connected, oriented paths is called a \emph{chain}.   For a path with $n$ edges in a triangulated planar shape,  $\partial_n$ defines a 1-chain.   The vertices on a 1-simplex (edge) $\sigma_i$ are the boundaries on $\sigma_i$.  In other words, the boundary of $n$ vertices $\left[v_0,\dots,v_n\right]$ is the $(n-1)$-chain formed by the sum of the faces~\cite[\S 2.1]{Hatcher2002CUPalgebraicTopology}.   For a 1-chain $c = \sum\lambda_i\sigma_i, \lambda_i\in \mathbb{Z}\mbox{mod}~2$ ({\em i.e.}, for an integer coefficient $\lambda_i$ in a 1-chain summand,  $\lambda_i$ mod~2 = 0 or 1),  the \emph{boundary} of the 1-chain is the sum of the boundaries of its 1-simplexes, namely, 
\[
\partial c = \lambda_1\partial\sigma_1 + \cdots + \lambda_n\partial\sigma_n= \mathop{\sum}\limits_{i=1}^n\lambda_i\partial\sigma_i.   
\]

Let $K$ be a simplicial complex and let $C_2(K), C_1(K), C_0(K)$ be an additive Abelian group of 2-chains, 1-chains and 0-chains, respectively. Consider a sequence of homomorphisms (boundary maps) of Abelian groups, namely,
\[
\cdots \mathop{\longrightarrow}\limits^{\partial_3} C_2 \mathop{\longrightarrow}\limits^{\partial_2} C_1\mathop{\longrightarrow}\limits^{\partial_1}
C_0\mathop{\longrightarrow}\limits^{\partial_0} 0.
\]
Let $C_1$ be a group of 1-chains of edges and let $C_0$ be a group of  0-chains of vertices.   In general, $p$-chains under addition form an Abelian group (denoted by $\left(C_p,+\right)$ or $C_p = C_p(K)$, when addition is understood).    Each member of $C_0$ is a 0-chain  (a linear combination of vertices) on the boundary of a 1-chain in $C_1$.   The kernel $\partial_1: C_1(K)\longrightarrow C_0(K)$ is a group denoted by $Z_1$.    Elements of $ \mbox{ker}\partial_1$ are called cycles. The image of $\partial_2
$  is the group $B_1 = B_1(K)$, which is a subgroup of $Z_1$.     Elements of $\mbox{img}\partial_2$ are called boundaries.   The quotient group $H_1 = Z_1/B_1 = \mbox{ker}\partial_1/\mbox{img}\partial_2$ isolates those cycles in $Z_1$ with empty boundaries.   Elements of $ H_1$ are called 1-cycles, {\em i.e.}, those cycles in $Z_1$ that are not boundaries.   From a quotient group perspective, elements of $H_1$ are cosets of $\mbox{img}\partial_2 = B_1$.


\begin{example}\label{ex:BettiNumber} {\bf Sample Cycles}.\\
For example, let edges $e_1,e_2,e_3,e_4,e_5,e_6,e_7$ and vertices $v_1,v_2,v_3,v_4,v_5,v_6$ on a triangulated shape (not shown)
be represented in Fig.~\ref{fig:chainMap}.   Then, we have
\begin{description}
\item[{\large $\boldsymbol{B_1}$}]  collection of boundaries written as 1-chains, {\em e.g.},\\ 
$\boldsymbol{\bullet}\partial(e_3,e_6,e_7) = \partial H_o = v_3 + v_4 - v_6$ is the boundary of the hole $H_o$ in Fig.~\ref{fig:chainMap}.
\item[{\large $\boldsymbol{Z_1}$}] collection of cycles written as 1-chains.   For simplicity, we consider only three cycles in $Z_1$ based on the labelled edges in Fig.~\ref{fig:chainMap}, namely,\\
$\boldsymbol{\bullet}\partial\left(e_1,e_2,e_3,e_4,e_5\right) = v_1 + v_2 + v_3 + v_4 + v_5 - v_5 - v_4 - v_3 - v_2 - v_1 = 0$.\\
$\boldsymbol{\bullet}\partial\left(e_1,e_2,e_7,e_6,e_4, e_5\right) = v_1 + v_2 + v_3 + v_6 + v_4 + v_5 - v_5 - v_4 - v_6 - v_4 - v_3 - v_2 - v_1 = 0$.\\  
$\boldsymbol{\bullet}\partial(e_3,e_6,e_7) = \partial H_o = v_3 + v_4 - v_6$ {\rm (}appears in $B_1${\rm )}.
\qquad \textcolor{blue}{\Squaresteel}
\end{description}
\end{example} 

\begin{remark}
With the quotient group $H_1$, we factor out of $Z_1$ the chains that are the hole boundaries in $B_1$.   From the features of the 1-cycles in homology groups $H_1$, we define a signature of a shape based on the description of 1-cycles, which is easily compared with the signatures of other shapes.  \qquad \textcolor{blue}{\Squaresteel}
\end{remark}
  
Let $(\mathcal{H}_1,\dnear)$ be a collection of 1-cycles on shape complexes equipped with a descriptive proximity $\dnear$~\cite[\S 4]{DiConcilio2016descriptiveProximity},~\cite[\S 1.8]{Peters2016ComputationalProximity}, based on the descriptive intersection $\dcap$ of nonempty sets $A$ and $B$~\cite[\S 3]{Peters2013mcs}.   With respect to 1-cycle sets of connected, oriented edges $e_1,e_2$ in $H_1$, for example, we consider $e_1\dcap e_2$.  For each given 1-cycle $A$ (denoted by $\cyc A$), find all 1-cycles $\cyc B$ in $\mathcal{H}_1$ that have nonempty descriptive intersection with $\cyc A$, {\em i.e.}, $\cyc A\ \dcap\ \cyc B\neq \emptyset$.  This results in a Leader uniform topology on $H_1$~\cite{Leader1959} and a main result in this paper.


Let $A\ \sn\ B$ be a strong proximity between nonempty sets $A$ and $B$, {\em i.e.}, $A$ and $B$ have nonempty intersection.  


\begin{theorem}\label{thm: LeaderUniformTopology}
Let $\left(\mathcal{H}_1, \left\{\sn,\dnear\right\}\right)$ be a collection of 1-dimensional homology groups $H_1$ equipped with a proximal relator  $\left\{\sn,\dnear\right\}$ and which is a collection of 1-cycles on a simplicial complex covering a finite, bounded planar shape and let 
\[
\Phi(\mathcal{H}_1) = \left\{\Phi(\cyc A):\mbox{1-cycle}\ \cyc A\in \mathcal{H}_1\right\}\ \mbox{{\rm (}{\bf Set of descriptions of $\boldsymbol{\cyc A\in \mathcal{H}_1}$}{\rm )}}
\]
be a set of descriptions $\Phi(\cyc A)$ of 1-cycles $\cyc A$ in $\mathcal{H}_1$.   A Leader uniform topology is derivable from $\Phi(\mathcal{H}_1)$.
\end{theorem}   


\section{Preliminaries}
This section briefly presents the basic approach to defining finite, bounded planar shape barcodes based on two useful proximities (strong spatial proximity $\sn$ and descriptive proximity $\dnear$).   A \emph{shape barcode} is a feature vector that describes a specific shape in terms of 1-cycle geometry, rank of $H_1$, characteristics of a homology nerve on $H_1$,  closure finiteness and 1-cycle arc characteristics based on a descriptive weak topology on $H_1$.   By \emph{proximity} of a pair of sets, we mean spatial closeness of the sets.   For a complete introduction to spatial proximity, see A. Di Concilio~\cite{Concilio2009} and the earlier overview of proximity by S.A. Naimpally and B.D. Warrack~\cite{Naimpally70withWarrack}.   A proximal hit-and-miss topology is a natural outcome of the traditional forms of proximity (see, {\em e.g.}, G. Beer~\cite[\S 2.2, p. 45]{Beer1993bk}).   By descriptive proximity of a pair of sets, we mean the closeness of the descriptions of the sets.   For a complete study of descriptive proximity, see A. Di Concilio, C. Guadagni, J.F. Peters and S. Ramanna~\cite{DiConcilio2016descriptiveProximity}.   In Section~\ref{sec:descriptiveCW}, a descriptive CW topology (Closure finite Weak topology) is defined for a collection $\mathcal{H}_1$ of homology groups $H_1$ equipped the descriptive proximity $\dnear$. 

\subsection{Basic Approach}
The basic approach in homology in classifying a finite, bounded planar shape $\sh A$ covered with a simplicial complex $K$ is to analyze a collection $\mathcal{H}_1$ of homology groups $H_1$ on $\sh A$, which is a set of 1-cycles.   A \emph{1-cycle} $A$ in $\mathcal{H}_1$ (denoted by $\cyc A$) is a simple, closed, connected path containing 1-simplexes (edges) that are not boundaries of holes in $\sh A$.   The story starts by identifying $1$-dimensional homology groups $Z_1$ ({\em i.e.}, groups whose members are cycles that are closed, connected paths on $1$-simplexes) and 1-dimensional groups $B_1$ containing cycles that are boundaries of holes.   From $Z_1$ and $B_1$, we then derive a homology group $H_1 = Z_1/B_1$ (a quotient group which factors out the cycle boundaries in $Z_1$) containing 1-cycles.

Notice that every planar shape has a distinguished 1-cycle, namely, the contour of a shape.   The features (distinguishable characteristics) of 1-cycles in $H_1$ provide a barcode for a particular shape $\sh A$, which is a feature vector in an $n$-dimensional Euclidean space $\mathbb{R}^n$.   A shape  $\sh A$ barcode describes $\sh A$ and is an instance of the signature of the shape (denoted by $\sig(\sh A)$).    In the study of a shape $\sh A$ that \emph{persists} and yet changes over time, the rank of $H_1$ is an important shape characteristic to include in the signature $\sig(\sh A)$.   In simple terms, the rank of $H_1$ is the number of 1-cycles in $H_1$~\cite[\S 2.2, p. 96]{Munch2017TDArank} on complex $K$ on a shape $\sh A$.
The \emph{rank} of $H_1$ (denoted by $rH_1$) is also called the Betti number of $H_1$.   Viewing the rank of $H_1$ in another way, the Betti number of $H_1$ is the number $\mathbb{Z}$ summands, when $H_1$ is written as the direct sum of its cyclic subgroups~\cite[\S 2.1, p. 1390]{Hatcher2002CUPalgebraicTopology}.  For example, the rank of $Z_1$ for Example~\ref{ex:BettiNumber} is 2.

\subsection {Framework for Two Recent Proximities}
This section briefly presents a framework for two recent types of proximities, namely, \emph{strong proximity} and the more recent \emph{descriptive proximity} in the study of computational proximity~\cite{Peters2016ComputationalProximity}.

Let $A$ be a nonempty set of vertices, $p\in A$ in a bounded region $X$ of the Euclidean plane.  An \emph{open ball} $B_r(p)$ with radius $r$ is defined by
\[
B_r(p) = \left\{q\in X: \norm{p - q} < r\right\}\ \mbox{(Open ball with center $p$, radius $r$)}.
\]
The \emph{closure} of $A$ (denoted by $\cl A$) is defined by
\[
\cl A = \left\{q\in X: B_r(q)\subset A\ \mbox{for some $r$}\right\}\ \mbox{(Closure of set $A$)}.
\]
The \emph{boundary} of $A$ (denoted by $\bdy A$) is defined by
\[
\bdy A = \left\{q\in X: B(q)\subset A\ \cap\ X\setminus A\right\}\ \mbox{(Boundary of set $A$)}.
\]
Of great interest in the study of shapes is the interior of a shape, found by subtracting the boundary of a shape from its closure.  In general, the \emph{interior} of a nonempty set $A\subset X$ (denoted by $\Int A$) defined by
\[
\Int A = \cl A - \bdy A\ \mbox{(Interior of set $A$)}.
\]


\setlength{\intextsep}{0pt}
\begin{wrapfigure}[13]{R}{0.42\textwidth}
\begin{minipage}{5.0 cm}
\centering
\begin{pspicture}
(-0.2,-1.5)(5,3)
\psline[linestyle=solid]%
(4,0)(3,-1)(1,-1)(0,0)
\psline[linecolor=blue,arrowsize=0pt 5]{->}(4,0)(3,-1)
\psline[linestyle=solid]%
(1,-1)
\pspolygon[fillstyle=solid,fillcolor=lightgray](3,1)(3,-1)(1.8,0.2)
\psdots[dotstyle=o, linewidth=1.2pt,linecolor = black, fillcolor = yellow]%
(0,0)(1,-1)(1,1)(1.8,0.2)(3,1)(3,-1)(1.8,0.2)(4,0)(4,2)(4.5,1.5)
\psdots[dotstyle=o, linewidth=1.2pt,linecolor = black, fillcolor = red]%
(3,1)(3,-1)
\psline[linecolor=blue,arrowsize=0pt 5]{->}(3,1)(4,2)
\psline[linecolor=blue,arrowsize=0pt 5]{->}(4,2)(4.5,1.5)
\psline[linecolor=blue,arrowsize=0pt 5]{->}(4.5,1.5)(4,0)
\rput(4.0,2.25){\textcolor{black}{\large $\boldsymbol{v_7}$}}\rput(4.75,1.5){\textcolor{black}{\large $\boldsymbol{v_8}$}}
\psline[linecolor=blue,arrowsize=0pt 5]{->}(0,0)(1,1)
\psline[linecolor=blue,arrowsize=0pt 5]{->}(1,1)(3,1)
\psline[linecolor=blue,arrowsize=0pt 5]{->}(3,1)(3,-1)
\psline[linecolor=blue,arrowsize=0pt 5]{->}(3,-1)(1,-1)
\psline[linecolor=blue,arrowsize=0pt 5]{->}(1,-1)(0,0)
\psline[linecolor=blue,arrowsize=0pt 5]{->}(2.98,-0.98)(1.8,0.2) 
\psline[linecolor=blue,arrowsize=0pt 5]{->}(1.81,0.21)(2.98,0.98) 
\rput(2.5,0.1){\textcolor{black}{\large $\boldsymbol{H_o}$}}
\rput(3,1.25){\textcolor{black}{\large $\boldsymbol{v_3}$}}\rput(1,1.25){\textcolor{black}{\large $\boldsymbol{v_2}$}}
\rput(-0.25,0){\textcolor{black}{\large $\boldsymbol{v_1}$}}\rput(3,-1.25){\textcolor{black}{\large $\boldsymbol{v_4}$}}
\rput(1,-1.25){\textcolor{black}{\large $\boldsymbol{v_5}$}}\rput(1.4,0.2){\textcolor{black}{\large $\boldsymbol{v_6}$}}
\rput(4.4,0.2){\textcolor{black}{\large $\boldsymbol{v_9}$}}
\rput(0.25,0.8){\textcolor{black}{\large $\boldsymbol{\cyc A}$}}
\rput(3.2,1.8){\textcolor{black}{\large $\boldsymbol{\cyc B}$}}
\end{pspicture}
\caption[]{$\boldsymbol{\cyc A\ \sn\ \cyc B}$}
\label{fig:sn}
\end{minipage}
\end{wrapfigure}

\emph{Proximities} are nearness relations.  In other words, a \emph{proximity} between nonempty sets is a mathematical expression that specifies the closeness of the sets.   A \emph{proximity space} results from endowing a nonempty set with one or more proximities.   Typically, a proximity space is endowed with a common proximity such as the proximities from \u Cech~\cite{Cech1966}, Efremovi\u c~\cite{Efremovic1952}, Lodato~\cite{Lodato1962}, and Wallman~\cite{Wallman1938}, or the more recent descriptive proximity~\cite{Peters2013mcsintro}.

\subsection{Strong Proximity}
Nonempty sets $A,B$ in a space $X$ equipped with the strong proximity $\sn$ are \emph{strongly near} [\emph{strongly contacted}] (denoted $A\ \sn\ B$), provided the sets have at least one point in common.   The strong contact relation $\sn$ was introduced in~\cite{Peters2015JangjeonMSstrongProximity} and axiomatized in~\cite{PetersGuadagni2015stronglyNear},~\cite[\S 6 Appendix]{Guadagni2015thesis} (see, also,~\cite[\S 1.5]{Peters2016ComputationalProximity},~\cite{Peters2015JangjeonMSstrongProximity,Peters2015AMSJmanifolds}) and elaborated in~\cite{Peters2016ComputationalProximity}.


Let $A, B, C \subset X$ and $x \in X$.  The relation $\sn$ on the family of subsets $2^X$ is a \emph{strong proximity}, provided it satisfies the following axioms.

\begin{description}
\item[{\rm\bf (snN0)}] $\emptyset\ \notfar\ A, \forall A \subset X $, and \ $X\ \sn\ A, \forall A \subset X$.
\item[{\rm\bf (snN1)}] $A\ \sn\ B \Leftrightarrow B\ \sn\ A$.
\item[{\rm\bf (snN2)}] $A\ \sn\ B$ implies $A\ \cap\ B\neq \emptyset$. 
\item[{\rm\bf (snN3)}] If $\{B_i\}_{i \in I}$ is an arbitrary family of subsets of $X$ and  $A\ \sn\ B_{i^*}$ for some $i^* \in I \ $ such that $\Int(B_{i^*})\neq \emptyset$, then $  \ A \sn (\bigcup_{i \in I} B_i)$ 
\item[{\rm\bf (snN4)}]  $\mbox{int}A\ \cap\ \mbox{int} B \neq \emptyset \Rightarrow A\ \sn\ B$.  
\qquad \textcolor{blue}{$\blacksquare$}
\end{description}

\noindent When we write $A\ \sn\ B$, we read $A$ is \emph{strongly near} $B$ ($A$ \emph{strongly contacts} $B$).  The notation $A\ \notfar\ B$ reads $A$ is not strongly near $B$ ($A$ does not \emph{strongly contact} $B$). For each \emph{strong proximity} (\emph{strong contact}), we assume the following relations:
\begin{description}
\item[{\rm\bf (snN5)}] $x \in \Int (A) \Rightarrow x\ \sn\ A$ 
\item[{\rm\bf (snN6)}] $\{x\}\ \sn\ \{y\}\ \Leftrightarrow x=y$  \qquad \textcolor{blue}{$\blacksquare$} 
\end{description}

For strong proximity of the nonempty intersection of interiors, we have that $A \sn B \Leftrightarrow \Int A \cap \Int B \neq \emptyset$ or either $A$ or $B$ is equal to $X$, provided $A$ and $B$ are not singletons; if $A = \{x\}$, then $x \in \Int(B)$, and if $B$ too is a singleton, then $x=y$. It turns out that if $A \subset X$ is an open set, then each point that belongs to $A$ is strongly near $A$.  The bottom line is that strongly near sets always share points, which is another way of saying that sets with strong contact have nonempty intersection.   

\begin{example}\label{ex:1cycles}
Assume that a finite, bounded shape $\sh A$ is covered by a simplicial complex containing 1-cycles $\cyc A,\cyc B$.   Let 1-cycle $\cyc A$ be represented by a sequence of vertices 
\[
v_1\rightarrow v_2\rightarrow v_3\rightarrow v_6\rightarrow v_4\rightarrow v_5\rightarrow v_1\ \mbox{($\cyc A$)},
\]
and let 1-cycle $\cyc B$ be represented by a sequence of vertices 
\[
v_3\rightarrow v_7\rightarrow v_8\rightarrow v_9\rightarrow v_4\rightarrow v_6\rightarrow v_3\ \mbox{($\cyc B$)},
\]
as shown in Fig.~\ref{fig:sn}.  Notice, for example, that the interior of 1-cycle $\cyc A$ includes the arc $\arc{v_3v_6}$, which is also in the interior of 1-cycle $cyc B$.  In this case, $\Int(\cyc A),\Int(\cyc B)$ have $\arc{v_3v_6}$ in common.   Hence, from axiom (snN4), $\cyc A\ \sn\ \cyc B$.
\qquad \textcolor{blue}{\Squaresteel}
\end{example}


\begin{definition}\label{def:homologyNerve}
Let $K$ be a simplicial complex covering a shape $\sh A$ and let $\mathcal{H}_1$ be the collection of 1-cycles in the homology groups on $K$.   A \emph{homology nerve} on $\mathcal{H}_1(K)$  (denoted by $\Nrv \mathcal{H}_1$) is defined by
\[
\Nrv \mathcal{H}_1 = \left\{\cyc A\in \mathcal{H}_1: \bigcap \cyc A\neq \emptyset\right\}\ \mbox{\rm ({\bf Homology Nerve}).\qquad \textcolor{blue}{\Squaresteel}}
\]
\end{definition}


The assumption made here is that every finite planar shape is bounded by a simple closed curve and has a nonempty interior. 
\begin{conjecture}
Every finite, bounded, planar shape with a decomposition and with at least one hole contains
a homology nerve that intersects with the boundary of a hole.
\qquad \textcolor{blue}{\Squaresteel}
\end{conjecture}
\begin{conjecture}
Every finite, bounded, planar shape with a decomposition and with at least one hole contains
a homology nerve that does not intersect with the boundary of any hole. 
\qquad \textcolor{blue}{\Squaresteel}
\end{conjecture}

\begin{remark} {\bf Short History of Topological Nerves}.\\
In topology, a nerve structure first appeared in 1926 in a paper on simplicial approximation by P. Alexandroff~\cite{Alexandroff1926MAnnNerfTheorem} and in 1932 in a monograph by P. Alexandroff~\cite[\S 33, p. 39]{Alexandroff1932elementaryConcepts}, elaborated by C. Kuratowski in 1933~\cite{Kuratowsk1933FMfundamentalNerveTheorem}.
Let the system of sets $F_1,\dots,F_s$  and system of vertices $v_1,\dots,v_s$ of a complex $K$ be related in such a way that the sets $F_{i_1},\dots,F_{r_F}$ have nonempty intersection if and only if the vertices $v_{i_0},\dots,v_{i_r}$ belong to $K$.   Then the complex $K$ is called the nerve of the system of sets in $K$.   A fundamental theorem concerning simplicial nerve complexes is given by B. Gr\"{u}nbaum in 1970~\cite{Grunbaum1970EMnerveComplexes}, namely, 
\begin{theorem}\label{thm: fundamentalNerveTheorem}
Each simplicial complex has the same homotopy type as its nerve.
\end{theorem}
\noindent Earlier, K. Borsuk obtained the following result in 1948.
\begin{theorem}\label{thm:regularDecompositions}~{\rm \cite[Cor. 2, p. 233]{Borsuk1948DecompositionsHomotopyType}}
Finite dimensional spaces admitting similar regular decompositions have necessarily the same homotopy type.
\end{theorem}
As a result of Theorem~\ref{thm:regularDecompositions}, K. Borsuk observed that (i) for every finite dimensional space with a regular decomposition, there exists a polytope with the same homotopy type and (ii) the notion of an Alexandroff nerve makes it possible to construct such a polytope~\cite[p. 233]{Borsuk1948DecompositionsHomotopyType}, leading to\\
\begin{corollary}~{\rm \cite[Cor. 3, p. 234]{Borsuk1948DecompositionsHomotopyType}}
If the simplicial complex $K$ is a geometrical realization of the nerve of a regular decomposition of a finite dimensional space $A$, then the space $A$ and the polytope $\abs{K}$ have the same homotopy type.
\end{corollary}
A more tractable view of a nerve, more amenable for computational topology, is given by H. Edelsbrunner and J.L. Harer~\cite[\S III.2, p. 59]{Edelsbrunner1999}.   Let $F$ be a finite collection of sets.   A nerve consists of all nonempty subcollections of $F$ (denoted by $\Nrv F$) whose sets have nonempty intersection, {\em i.e.},
\[
\Nrv F = \left\{X\subseteq F: \bigcap X\neq \emptyset\right\}\ \mbox{(Edelsbrunner-Harer Nerve)}.
\]
A nerve is an example of an abstract simplicial complex, regardless of the sets in $F$.   Strongly proximal Edelsbrunner-Harer nerves were introduced in 2016 by J.F. Peters and E. \.{I}nan~\cite{PetersInan2016spNerves}.    Nerve spoke complexes (useful for nerves on Vorono\"{i} tessellations) are introduced in J.F. Peters~\cite{Peters2017arXiv1704-05909spokes}.  An overview of recent work on nerve complexes is given by H. Dao, J. Doolittle, K. Duna, B. Goeckner, B. Holmes and J. Lyle~\cite{Dao2017arXivNerveComplexes}.
\qquad \textcolor{blue}{\Squaresteel}
\end{remark}


\begin{example}
Assume that a shape $\sh A$ is covered by a simplicial complex with homology groups $H_1$ containing 1-cycles $\cyc A,\cyc B$ from Example~\ref{ex:1cycles}.   Hence,
\[
\Nrv \mathcal{H}_1 = \left\{\cyc A,\cyc B\right\}\ \mbox{\rm ({\bf Sample homology nerve}). \qquad \textcolor{blue}{\Squaresteel}}
\]
\end{example}

\begin{lemma}\label{prop:2spoke}
Let homology groups $H_1$ contain 1-cycles $\cyc A, \cyc B$ on complex $K$ covering shape $\sh A$.   Then
$\cyc A\ \sn\ \cyc B \Rightarrow \cyc A\ \cap\ \cyc B \neq \emptyset$, if and only if $\cyc A, \cyc B\in \Nrv \mathcal{H}_1$ for homology nerve complex $\Nrv \mathcal{H}_1\in 2^{\mathcal{H}_1}$ .
\end{lemma}
\begin{proof}
$\cyc A\ \sn\ \cyc B \Rightarrow \cyc A\ \cap\ \cyc B\neq \emptyset$ (from (snN2)) $\Leftrightarrow$
$\cyc A, \cyc B\in \Nrv \mathcal{H}_1$ (from Def.~\ref{def:homologyNerve}) for at least one nerve complex $\Nrv \mathcal{H}_1\in 2^{\mathcal{H}_1}$ .
\end{proof}

\begin{lemma}\label{prop:H1}
Let $\Nrv_1 \mathcal{H}_1,\Nrv_2 \mathcal{H}_1$ be homology nerves for homology groups $H_1$ on complex $K$ covering shape $\sh A$.  Then
$\Nrv_1 \mathcal{H}_1\ \sn\ \Nrv_2 \mathcal{H}_1$ implies $\Nrv_1 \mathcal{H}_1\ \cap\ \Nrv_2 \mathcal{H}_1\neq \emptyset$  for some $\cyc A\in \Nrv_1 \mathcal{H}_1$ and $\cyc B\in \Nrv_2 \mathcal{H}_1$.
\end{lemma}
\begin{proof}
$\Nrv_1 \mathcal{H}_1\ \sn\ \Nrv_2 \mathcal{H}_1 \Rightarrow \Nrv_1 \mathcal{H}_1\ \cap\ \Nrv_2 \mathcal{H}_1\neq \emptyset$ (from (snN2)).   Consequently, 
$\cyc A\ \sn\ \cyc B \rightarrow \cyc A\ \cap\ \cyc B\neq \emptyset$ (from Lemma~\ref{prop:2spoke}) for at least one
$\cyc A\in \Nrv \mathcal{H}_1$ and for at least one $\cyc B\in \Nrv \mathcal{H}_2$, since a homology nerve is a set of 1-cycles  (from Def.~\ref{def:homologyNerve}).
\end{proof}

\begin{theorem}\label{thm:H1}
Let $\Nrv_1 \mathcal{H}_1,\Nrv_2 \mathcal{H}_1$ be homology nerves for homology groups $H_1$ on a simplicial complex covering shape $\sh A$.  
$\Nrv_1 \mathcal{H}_1\ \sn\ \Nrv_2 \mathcal{H}_1$ if and only if  $\cyc A\ \sn\ \cyc B$ for some $\cyc A\in \Nrv_1 \mathcal{H}_1$ and $\cyc B\in \Nrv_2 \mathcal{H}_1$.
\end{theorem}
\begin{proof}
Immediate from Lemma~\ref{prop:H1}.
\end{proof}

\begin{corollary}
Let $\Nrv_1 \mathcal{H}_1,\Nrv_2 \mathcal{H}_1,\Nrv_3 \mathcal{H}_1$ be homology nerves for homology groups $H_1$ on a simplicial complex covering shape $\sh A$.
If $(\Nrv_1 \mathcal{H}_1\ \cup\ \Nrv_2 \mathcal{H}_1) \sn \Nrv_3 \mathcal{H}_1$, then $\Nrv_1 \mathcal{H}_1\ \sn\ \Nrv_3 \mathcal{H}_1$ or $\Nrv_2 \mathcal{H}_1\ \sn\ \Nrv_3 \mathcal{H}_1$ for the three homology nerves $\Nrv_1 \mathcal{H}_1,\Nrv_2 \mathcal{H}_1,\Nrv_3 \mathcal{H}_1$ on the homology groups $H_1$.
\end{corollary}
\begin{proof}
From Lemma~\ref{prop:H1}, $(\Nrv_1 \mathcal{H}_1\ \cup\ \Nrv_2 \mathcal{H}_1) \sn \Nrv_3 \mathcal{H}_1\Rightarrow (\Nrv_1 \mathcal{H}_1\ \cup\ \Nrv_2 \mathcal{H}_1)\ \cap\ \Nrv_2 \mathcal{H}_1\neq \emptyset$.   And, from Theorem~\ref{thm:H1},  $(\Nrv_1 \mathcal{H}_1\ \cup\ \Nrv_2 \mathcal{H}_1) \sn \Nrv_3 \mathcal{H}_1$ if and only if  $\cyc A\ \sn\ \cyc B$ for some $\cyc A\in (\Nrv_1 \mathcal{H}_1\ \cup\ \Nrv_2 \mathcal{H}_1)$ and $\cyc B\in \Nrv_3 \mathcal{H}_1$.   Hence, $\Nrv_1 \mathcal{H}_1\ \sn\ \Nrv_3 \mathcal{H}_1$ or $\Nrv_2 \mathcal{H}_1\ \sn\ \Nrv_3 \mathcal{H}_1$.
\end{proof}

\begin{figure}[!ht]
\centering
\begin{pspicture}
(-0.5,-2.5)(11,3)
\psframe[linewidth=2pt,framearc=.3](-0.8,-2.5)(11.8,2.5)
\rput(0.0,2.0){\textcolor{black}{\large $\boldsymbol{H_1}$}}
\pscurve[linecolor=black,arrowscale=2]{->}(4,0)(3,-1)(1,-1)(0,0)
\pscurve[linecolor=blue,arrowscale=2]{->}(0,0)(1,1)(3,1)(4,0)
\pscurve[linecolor=black,arrowscale=2]{<-}(3,1)(4,2)(4.5,1.5)
\pscurve[linecolor=black,arrowscale=2]{->}(4.5,1.5)(4,0)(1.8,0.2)(3,1)
\psdots[dotstyle=o, linewidth=1.2pt,linecolor = black, fillcolor = yellow]%
(0,0)(1,-1)(1,1)(1.8,0.2)(3,1)(3,-1)(1.8,0.2)(4,0)(4,2)(4.5,1.5)
\psdots[dotstyle=o, linewidth=1.2pt,linecolor = black, fillcolor = red]%
(3,1)(4,0)
\rput(4.0,2.25){\textcolor{black}{\large $\boldsymbol{v_7}$}}
\rput(4.75,1.5){\textcolor{black}{\large $\boldsymbol{v_8}$}}
\rput(2.7,1.25){\textcolor{black}{\large $\boldsymbol{v_3}$}}\rput(1,1.25){\textcolor{black}{\large $\boldsymbol{v_2}$}}
\rput(-0.25,0){\textcolor{black}{\large $\boldsymbol{v_1}$}}\rput(3,-1.25){\textcolor{black}{\large $\boldsymbol{v_4}$}}
\rput(0.0,0.75){\textcolor{black}{\large $\boldsymbol{\cyc A}$}}\rput(4.8,2.0){\textcolor{black}{\large $\boldsymbol{\cyc B}$}}
\rput(1,-1.25){\textcolor{black}{\large $\boldsymbol{v_5}$}}\rput(1.4,0.2){\textcolor{black}{\large $\boldsymbol{v_{9}}$}}
\rput(4.35,0.0){\textcolor{black}{\large $\boldsymbol{v_6}$}}
\pscurve[linecolor=black,arrowscale=2]{->}(10,0)(9,-1)(7,-1)(6,0)
\pscurve[linecolor=blue,arrowscale=2]{->}(6,0)(7,1)(9,1)(10,0)
\pscurve[linecolor=black,arrowscale=2]{->}(10,0)(10.5,-1.5)(10,-2)(9,-1)(9,1)
\psdots[dotstyle=o, linewidth=1.2pt,linecolor = black, fillcolor = yellow]%
(6,0)(7,1.0)(9,-1)(7,-1)(10,-2)(10.5,-1.5)
\psdots[dotstyle=o, linewidth=1.2pt,linecolor = black, fillcolor = red]%
(10,0)(9,1)
\rput(5.62,0){\textcolor{black}{\large $\boldsymbol{v_{10}}$}}\rput(7,1.25){\textcolor{black}{\large $\boldsymbol{v_{11}}$}}
\rput(9,1.25){\textcolor{black}{\large $\boldsymbol{v_{12}}$}}\rput(10.35,0.0){\textcolor{black}{\large $\boldsymbol{v_{13}}$}}
\rput(8.7,-1.265){\textcolor{black}{\large $\boldsymbol{v_{14}}$}}\rput(7,-1.25){\textcolor{black}{\large $\boldsymbol{v_{15}}$}}
\rput(6,0.75){\textcolor{black}{\large $\boldsymbol{\cyc C}$}}\rput(10.95,-0.8){\textcolor{black}{\large $\boldsymbol{\cyc D}$}}
\rput(10.75,-1.75){\textcolor{black}{\large $\boldsymbol{v_{17}}$}}\rput(10,-2.25){\textcolor{black}{\large $\boldsymbol{v_{16}}$}}
\end{pspicture}
\caption[]{$\cyc A\ \sn\ \cyc B$ and $\cyc A\ \dnear\ \cyc C$}
\label{fig:nearCycles}
\end{figure}

\subsection{Descriptive Proximity}\label{sec:dnear}
In the run-up to a close look at extracting features from shape complexes, we first consider descriptive proximities introduced in~\cite{Peters2013mcsintro}, fully covered in~\cite{DiConcilio2016descriptiveProximity} and briefly introduced, here.   
Descriptive proximities resulted from the introduction of the descriptive intersection of pairs of nonempty sets~\cite{Peters2013mcsintro},~\cite[\S 4.3, p. 84]{Naimpally2013}. 

\begin{description}
\item[{\rm\bf ($\boldsymbol{\Phi}$)}] $\Phi(A) = \left\{\Phi(x)\in\mathbb{R}^n: x\in A\right\}$, set of feature vectors.
\item[{\rm\bf ($\boldsymbol{\dcap}$)}]  $A\ \dcap\ B = \left\{x\in A\cup B: \Phi(x)\in \Phi(A) \& \in \Phi(x)\in \Phi(B)\right\}$.
\qquad \textcolor{blue}{$\blacksquare$}
\end{description}


Let $\Phi(x)$ be a feature vector for an arc $x$ in a simplicial complex on a planar shape.   For example, let $\Phi(x)$ be a feature vector representing single arc feature such as a Fourier descriptor in measuring the difference between arcs in a complex~\cite{PersoonFu1986PAMIFourierDescriptors} or uniform iso-curvature of arc along a curved edge~\cite[\S 2.2]{BenkhlifaGhorbel2017isoCurvature}.  For simplicity, we limit the description of an arc to the uniform iso-curvature of the arc between vertices in the curved edges of a 1-cycle such as those shown in Fig.~\ref{fig:nearCycles}.  $A\ \delta_{\Phi}\ B$ reads $A$ is descriptively near $B$, provided $\Phi(x) = \Phi(y)$ for at least one pair of points, $x\in A, y\in B$.  The proximity $\delta$ in the \u{C}ech, Efremovi\u c, and Wallman proximities is replaced by $\dnear$, which satisfies the following Descriptive Lodato Axioms from~\cite[\S 4.15.2]{Peters2014book}.

\begin{description}
\item[{\rm\bf (dP0)}] $\emptyset\ \dfar\ A, \forall A \subset X $.
\item[{\rm\bf (dP1)}] $A\ \dnear\ B \Leftrightarrow B\ \dnear\ A$.
\item[{\rm\bf (dP2)}] $A\ \dcap\ B \neq \emptyset \Rightarrow\ A\ \dnear\ B$.
\item[{\rm\bf (dP3)}] $A\ \dnear\ (B \cup C) \Leftrightarrow A\ \dnear\ B $ or $A\ \dnear\ C$.
\item[{\rm\bf (dP4)}] $A\ \dnear\ B$ and $\{b\}\ \dnear\ C$ for each $b \in B \ \Rightarrow A\ \dnear\ C$\ \mbox{({\bf Descriptive Lodato})}. \qquad \textcolor{blue}{$\blacksquare$}
\end{description}



\begin{proposition}\label{prop:dnear}{\rm~\cite[\S 2.2]{Peters2017arXiv1708-04147planarShapes}}
Let $\left(X,\dnear\right)$ be a descriptive proximity space, $A,B\subset X$.  Then $A\ \dnear\ B \Rightarrow A\ \dcap\ B\neq \emptyset$.
\end{proposition}
\begin{proof}
See~\cite[\S 2.2]{Peters2017arXiv1708-04147planarShapes} for the proof.
\end{proof}

Next, consider a proximal form of a Sz\'{a}z relator~\cite{Szaz1987}.  A \emph{proximal relator} $\mathscr{R}$ is a set of relations on a nonempty set $X$~\cite{Peters2016relator}.  The pair $\left(X,\mathscr{R}\right)$ is a proximal relator space.  The connection between $\sn$ and $\near$ is summarized in Prop.~\ref{thm:sn-implies-near}.

\begin{lemma}\label{thm:sn-implies-near}{\rm~\cite[\S 2.2]{Peters2017arXiv1708-04147planarShapes}}
Let $\left(X,\left\{\dnear,\sn\right\}\right)$ be a proximal relator space, $A,B\subset X$.  Then 
$A\ \sn\ B \Rightarrow A\ \dnear\ B$.
\end{lemma}
\begin{proof}
See~\cite[\S 2.2]{Peters2017arXiv1708-04147planarShapes} for the proof.
\end{proof}

\begin{example}{\bf Descriptively Near 1-Cycles in $\boldsymbol{H_1}$}.
Let $\cyc A, \cyc B, \cyc C, \cyc D$ in Fig.~\ref{fig:nearCycles} be 1-cycles in a collection of homology groups $\mathcal{H}_1$ on a simplicial complex covering a planar shape.    Further, for example, let 
\[
\Phi(\cyc A) = \left\{\Phi(\arc{vv'})\in \cyc A: \Phi(\arc{vv'}) =\ \mbox{uniform iso-curvature of}\ \arc{vv'}\right\}.
\]
Let $\mathcal{H}_1$ be equipped with the relator $\left\{\sn,\dnear\right\}$. Then observe
\begin{compactenum}[1$^o$]
\item In Fig.~\ref{fig:nearCycles}, 
\begin{align*}  
\cyc A\ \mbox{has}\  &\overbrace{[v_1,v_2,v_3,v_4,v_5,v_6]}^{\textcolor{blue}{\mbox{$\cyc A$\ vertices}}}:
                                                                 v_1\rightarrow v_2\rightarrow v_3\rightarrow v_6\rightarrow v_4\rightarrow v_5\rightarrow v_1.\\
\cyc B\ \mbox{has}\   &\overbrace{[v_3,v_6,v_7,v_8]}^{\textcolor{blue}{\mbox{$\cyc B$\ vertices}}}:
                                                                 v_3\rightarrow v_6\rightarrow v_8\rightarrow v_7\rightarrow v_3.\\
\mbox{edge}\ \arc{v_3v_6} &\ \in \Int(\cyc A)\ \mbox{and}\ \arc{v_3v_6}\ \in \Int(\cyc B), \mbox{\em i.e.},\Int(\cyc A)\cap\Int(\cyc B)\neq\emptyset.
\end{align*}
Consequently, from Axiom (snN4), $\cyc A\ \sn\ \cyc B$ and from Lemma~\ref{thm:sn-implies-near}, $\cyc A\ \dnear\ \cyc B$.   Hence, from Proposition~\ref{prop:dnear}, $\cyc A\ \dcap\ \cyc B\neq \emptyset$.
\item In Fig.~\ref{fig:nearCycles}, 
\begin{align*} 
\cyc C\ \mbox{has}\  & \overbrace{[v_{10},v_{11},v_{12},v_{13},v_{14},v_{15}]}^{\textcolor{blue}{\mbox{$\cyc C$\ vertices}}}:\\
                                                                 &v_{10}\rightarrow v_{11}\rightarrow v_{12}\rightarrow v_{13}\rightarrow v_{14}\rightarrow v_{15}\rightarrow v_{10}.\qquad\qquad\\ 
\cyc D\ \mbox{has}\   &\overbrace{[v_{12},v_{13},v_{17},v_{16},v_{14}]}^{\textcolor{blue}{\mbox{$\cyc D$\ vertices}}}:
                                                                 v_{12}\rightarrow v_{13}\rightarrow v_{17}\rightarrow v_{16}\rightarrow v_{14}\rightarrow v_{12}.\\  
\cyc A\ \dnear\ \cyc C\ \mbox{, since}\  &\overbrace{\Phi(\arc{v_3v_6}) = \Phi(\arc{v_{12},v_{13}}).}^{\textcolor{blue}{\mbox{arcs have matching uniform iso-curvature}}}
\end{align*}
Consequently, $\cyc A\ \dcap\ \cyc C\neq \emptyset$.  From Axiom (dP2), $\cyc A\ \dnear\ \cyc C$.    Hence, from Proposition~\ref{prop:dnear}, the converse also holds, {\em i.e.}, 
\[
\cyc A\ \dnear\ \cyc C \Rightarrow \cyc A\ \dcap\ \cyc C\neq \emptyset.
\]
In other words, the 1-cycles $\cyc A, \cyc C$ in homology groups $\mathcal{H}_1$ represented in Fig.~\ref{fig:nearCycles} have descriptive proximity, since $\cyc A, \cyc C$ have curved edges with the same uniform iso-curvature.
\qquad \textcolor{blue}{\Squaresteel}
\end{compactenum}
\end{example}

 Let $2^{2^{\mathcal{H}_1}}$ denote a collection of sub-collections of 1-cycles of  $\mathcal{H}_1$.

\begin{theorem}\label{thm:spoke}
Let $\left(\mathcal{H}_1,\left\{\dnear,\sn\right\}\right)$ be a collection of homology groups endowed with a proximal relator and let 1-cycles $\cyc A, \cyc B\in \mathcal{H}_1$, homology nerves $\Nrv_1 \mathcal{H}_1,\Nrv_2 \mathcal{H}_1\in 2^{2^{\mathcal{H}_1}}$.  Then
\begin{compactenum}[1$^o$]
\item $\Nrv_1 \mathcal{H}_1\ \sn\ \Nrv_2 \mathcal{H}_1$ implies $\Nrv_1 \mathcal{H}_1\ \dnear\ \Nrv_2 \mathcal{H}_1$.
\item A 1-cycle $\cyc A\ \in\ \Nrv_1 \mathcal{H}_1\cap \Nrv_2 \mathcal{H}_1$ implies $\cyc A\ \in\ \Nrv_1 \mathcal{H}_1\ \dcap\ \Nrv_2 \mathcal{H}_1$.
\item $\cyc A\ \sn\ \cyc B\rightarrow \cyc A\ \dnear\ \cyc B$.
\end{compactenum}
\end{theorem}
\begin{proof}$\mbox{}$\\
1$^o$: Immediate from Lemma~\ref{thm:sn-implies-near}.\\
2$^o$: Let $cyc A\in \mathcal{H}_1$.   $\cyc A\ \in\ \Nrv_1 \mathcal{H}_1\cap \Nrv_2 \mathcal{H}_1$, provided $\Nrv_1 \mathcal{H}_1\ \sn\ \Nrv_2 \mathcal{H}_1$.   Then  $\cyc A\ \in\ \Nrv_1 \mathcal{H}_1\ \dcap\ \Nrv_2 \mathcal{H}_1$.   Hence, from  Prop.~\ref{prop:2spoke},  $\Nrv_1 \mathcal{H}_1\ \dnear\ \Nrv_2 \mathcal{H}_1$.\\
3$^o$: Immediate from Lemma~\ref{thm:sn-implies-near}.
\end{proof}

\begin{corollary}
Let $\left(\mathcal{H}_1,\left\{\dnear,\sn\right\}\right)$ be a collection of homology groups endowed with proximal relator, homology nerves $\Nrv_1 \mathcal{H}_1,\Nrv_2 \mathcal{H}_1\in 2^{2^{\mathcal{H}_1}}$ with $\Nrv_2 \mathcal{H}_1$ on shape $\sh B$.  Then
\begin{compactenum}[1$^o$]
\item $\Nrv_1 \mathcal{H}_1\ \sn\ \Nrv_2 \mathcal{H}_1$ implies $\Nrv_1 \mathcal{H}_1\ \dnear\ \sh B$.
\item $\Nrv_1 \mathcal{H}_1\cap \Nrv_2 \mathcal{H}_1\neq \emptyset$ implies $\Nrv_1 \mathcal{H}_1\ \dcap\ \Nrv_2 \mathcal{H}_1$.
\end{compactenum}
\end{corollary}

\subsection{Descriptive Homology Nerves and Shape Signature}\label{sec:descriptiveCW}
This section introduces descriptive homology nerves and the components of a shape signature.

\begin{definition}\label{def:descriptiveHomologyNerve}
Let $K$ be a simplicial complex covering a shape $\sh A$ and let $\mathcal{H}_1$ be the collection of 1-cycles in homology groups $H_1$ on $K$.   A \emph{descriptive homology nerve} on $\mathcal{H}_1(K)$  (denoted by $\Nrv_{\Phi} \mathcal{H}_1$) is defined by
\[
\Nrv_{\Phi} \mathcal{H}_1= \left\{\cyc A\in \mathcal{H}_1: \Dcap \cyc A\neq \emptyset\right\}\ \mbox{\rm ({\bf Descriptive Homology Nerve}).\qquad \textcolor{blue}{\Squaresteel}}
\]
The \emph{nucleus} of a descriptive homology nerve is any member $\cyc A\in \Nrv_{\Phi} \mathcal{H}_1$ that serves as a representative of the nerve inasmuch as $\cyc A$ defines a cluster $X$ that contains all 1-cycles $\cyc B$ such that $\cyc A\ \dnear\ \cyc B$.
\qquad \textcolor{blue}{\Squaresteel}
\end{definition}


\begin{theorem}\label{thm:descriptiveHomologyNerve}
Let $K$ be a simplicial complex covering a finite, bounded planar shape, $\mathcal{H}_1$ a collection of homology groups on $K$, and $\Phi(\mathcal{H}_1)$ a set of descriptions of the 1-cycles in $\mathcal{H}_1$. 
Every member of  $\Phi(\mathcal{H}_1)$ is the nucleus of a descriptive homology nerve $\Nrv_{\Phi} \mathcal{H}_1$.
\end{theorem}
\begin{proof}$\mbox{}$\\
By definition, $\Phi(\mathcal{H}_1) = \left\{\Phi(\cyc A): \cyc A\in \mathcal{H}_1\right\}$.
Let $\Phi(\cyc A)\in\Phi(\mathcal{H}_1)$.   Since $\cyc A\ \dnear\ \cyc A$, then, from Def.~\ref{def:descriptiveHomologyNerve},
$\cyc A$ is the nucleus of a descriptive nerve $\Nrv_{\Phi} \mathcal{H}_1$ containing one cycle, namely, $\cyc A$.    Let
\[
X = \left\{\cyc B\in \mathcal{H}_1: \cyc A\ \dnear\ \cyc B\neq \emptyset\right\}.
\]
Hence,  by  Def.~\ref{def:descriptiveHomologyNerve}, $X$ is a descriptive homology nerve and $\cyc A$ is the nucleus of the nerve $X$, {\em i.e.}, $\cyc A\ \dnear\ \cyc B$ for every member $\cyc B\in X$.  
\end{proof}

\begin{example}$\mbox{}$\\
Let the collection of homology groups $\mathcal{H}_1$ be represented the 1-cycles $\cyc A, \cyc B, \cyc C, \cyc D$ in Fig.~\ref{fig:nearCycles}.   Let uniform iso-curvature be used to describe a 1-cycle in $\mathcal{H}_1$.   Notice that curved edge $\arc{v_3v_6}\in \cyc B$ has the same uniform iso-curvature as $\arc{v_{12}v_{13}}\in \cyc C$ and $\arc{v_{12}v_{13}}\in \cyc D$.    Hence,
\[
\Nrv_{\Phi} \mathcal{H}_1= \left\{\cyc B, \cyc C, \cyc D\in \mathcal{H}_1\right\}\ \mbox{\rm ({\bf Descriptive Homology Nerve})},
\]
since
\[
\mathop{\bigcap}\limits_{\substack{\cyc X\in\\
                                                        \left\{\cyc B, \cyc C, \cyc D\right\}}}\cyc X\neq \emptyset.
\]
From Theorem~\ref{thm:descriptiveHomologyNerve}, $\cyc B$ is the nucleus of $\Nrv_{\Phi} \mathcal{H}_1$.
\qquad \textcolor{blue}{\Squaresteel}
\end{example}

\begin{conjecture}
Every finite, bounded, planar shape with a decomposition and with at least one hole contains
a descriptive homology nerve that intersects with the boundary of a hole.
\qquad \textcolor{blue}{\Squaresteel}
\end{conjecture}
\begin{conjecture}
Every finite, bounded, planar shape with a decomposition and with at least one hole contains
a descriptive homology nerve that does not intersect with the boundary of any hole. 
\qquad \textcolor{blue}{\Squaresteel}
\end{conjecture}

\noindent Consider next a basis for a shape signature.

\begin{definition}\label{def:signature} {\bf Shape Signature}.\\
Let $\mathcal{H}_1$ be a collection of homology groups on a simplicial complex covering a shape $\sh A$, a finite bounded planar region with nonempty interior and let $\Nrv_1 \mathcal{H}_1,\Nrv_2 \mathcal{H}_1\in 2^{{\mathcal{H}_1}}$.   Assume that $\mathcal{H}_1$ is equipped with a proximal relator $\left\{\sn, \dnear\right\}$. A signature of shape $\sh A$ {\rm (}denoted by $\sig(\sh A)${\rm )} is a feature vector that includes at least one of the following components.


\begin{compactenum}[1$^o$]
\item {\bf Geometry}:  One or more features of the curvature of each 1-cycle $\cyc A\in \mathcal{H}_1$ are included in $\sig(\sh A)$ that describes shape $\sh A$.
\item {\bf Homology}:  rank of the homology group $H_1$ (denoted by $rH_1$), {\em i.e.}, number of 1-cycle generators of $H_1$ is defined in terms of the rank of the cycles group $Z_1$ (denoted by $rZ_1$) and the rank of the boundaries group $B_1$ (denoted by $rB_1$) .  Recall that 
\[
rH_1 = r(Z_1/B_1) = rZ_1 - rB_1\ \mbox{({\bf Rank of a homology group})}\mbox{~\cite[p. 63]{Vick1994homologyTheory}}.
\]
The rank $rH_1$ (a Betti number) can change over time and provides a useful in indicator of planar shape persistence.   Hence, its inclusion in a shape $\sh A$ signature $\sig(\sh A)$ (barcode) is important in considering the persistent topology of data such as that found in R. Ghrist~\cite{Ghrist2008BAMSbarcodePersistence}.
\item {\bf Homology Nerve}:  Since every $\cyc A\in \mathcal{H}_1$ is the nucleus of a descriptive homology nerve $\Nrv_{\Phi} \mathcal{H}_1$ {\rm(}from Theorem~\ref{thm:descriptiveHomologyNerve}{\rm)},  select a component of $\Phi(\cyc A)$ (call it $x$) with a description that matches the description of the same component in the other members of $\Nrv_{\Phi} \mathcal{H}_1$.   Include $\Phi(x)$ in the signature of $\sh A$, {\em i.e.}, 
\[
\sig(\sh A) = \left(\dots,\Phi(x)\dots\right)\ \mbox{($\Phi(x)$ in feature vector that describes $\sh A$)}.
\]
\item {\bf Closure Finiteness}:
 Let $\arc{vv'}$ be an arc in a 1-cycle $\cyc A\in \mathcal{H}_1$ and $\cl(\arc{vv'})$ intersects only a finite number of other arcs in $\mathcal{H}_1$.   $\cl(\arc{vv'})$ is the closure of an arc in $\cyc A\ \dcap\ \cyc B$ for a finite number of 1-cycles.   For $\cyc A,\cyc B\in \mathcal{H}_1$, choose $\Phi(\cl(\arc{vv'}))\in \sig(\sh A)$ or $\Phi(\cyc A)\in \sig(\sh A)$ for a selected number of  1-cycles in $\mathcal{H}_1$.
\item {\bf descriptive CW}:  ({{\em i.e.}, \bf descriptive Weak Topology})  Assume that Closure Finiteness holds for the collection of homology groups $\mathcal{H}_1$ equipped with the descriptive proximity $\dnear$.   Let $\arc{vv'}$ be an arc in $H_1\in \mathcal{H}_1$ and let 1-cycle $\cyc A\in \mathcal{H}_1$.   Then $\cyc A$ is closed in $\mathcal{H}_1$, provided $\cyc A\cap \arc{vv'}\neq \emptyset$ is also closed in $H_1$ .   Then $\cyc A\ \sn\ \arc{vv'}$.   Hence, from Lemma~\ref{thm:sn-implies-near}, $\cyc A\ \dnear\ \arc{vv'}$.   For example, 1-cycles $\cyc A, \cyc B$ in Fig.~\ref{fig:nearCycles} overlap, since arc $\arc{v_3v_6}$ is common to both 1-cycles.   Such arcs provide an incisive feature for a shape signature.   Then, for a shape $\sh A$, include the description of such arcs in the shape signature $\sig(\sh A)$.   
\qquad \textcolor{blue}{\Squaresteel}
\end{compactenum}
\end{definition}

\begin{remark}
The original idea of a CW topology ({\bf C}losure finite {\bf W}eak topology) was to shift from structures in simplicial complexes $K$ that are the focus in P. Alexandroff~\cite{Alexandroff1932elementaryConcepts} and in P. Alexandroff, H. Hopf~\cite{Alexandroff1935} to homological structures called cells and cell complexes ({\em e.g.}, 0-cells (vertices) and 1-cells (open arcs) attached to a shape skeleton via maps to obtain a cell complex) in a homology on $K$~\cite[p. 214]{Whitehead1949BAMS-CWtopology}.   A \emph{cell complex} is a finite collection of cells~\cite{Hatcher2002CUPalgebraicTopology}. With a descriptive CW, we shift from a description of structures ({\em e.g.}, simplicial nerves~\cite[p. 2]{Peters2017arXiv1704-05909spokes} and nerve spokes~\cite[\S 2.2, p. 4]{Peters2017arXiv1704-05909spokes}~\cite[Def. 9, p. 8]{AhmadPeters2017arXiv1706-04549v1spokeComplexes}) in simplicial complexes to a description of structures such as homology nerves, collections of 1-cycles and overlapping arcs in a collection of homology groups $\mathcal{H}_1$ in cell complexes on finite bounded planar shapes.   Basically, with a descriptive CW on $\mathcal{H}_1$, we include those features of arcs, 1-cycles and homology nerves in $\mathcal{H}_1$ that provide a complete signature $\sig(\sh A)$ for a shape $\sh A$.   The motivation for doing this is an interest in measuring the persistence of the feature values of arcs, 1-cycles and homology nerves in homology groups over time.   This descriptive CW is based on the Closure finiteness and Weak topology axioms for a traditional CW complex given by K. J\"{a}nich~\cite[\S VII.3, p. 95]{Janich2005closedWeakTopology} founded on its original introduction by J.H.C. Whitehead~\cite{Whitehead1949BAMS-CWtopology}.
\qquad \textcolor{blue}{\Squaresteel}
\end{remark}

\section{Main Results}

\begin{theorem}
Every finite, bounded planar shape $\sh A$ covered by a simplicial complex  has a signature derived from the homology group on the complex.
\end{theorem}
\begin{proof}
From Def.~\ref{def:signature}, it is enough to include the rank of $H_1$ in  $\sig(\sh A)$ for a shape $\sh A$ to have a signature.
\end{proof}

\begin{lemma}\label{lemma:LeaderUniformTopologyOnH1}
Let $\mathcal{H}_1$ be a collection of homology groups equipped with the proximal relator $\mathscr{R}_{\Phi}=\left\{\sn,\dnear\right\}$ on a simplicial complex covering a finite, bounded shape.
Every collection of 1-dimensional homology groups $H_1\in \mathcal{H}_1$ endowed with the proximal relator $\mathscr{R}_{\Phi}$ defines a descriptive uniform Leader topology on $\mathcal{H}_1$.
\end{lemma}
\begin{proof}$\mbox{}$\\
\noindent The basic approach in this proof is to use the steps for constructing a uniform topology introduced by S. Leader~\cite{Leader1959} in constructing a descriptive uniform topology. \\
\noindent $\dcap$:  For each $\Nrv_1\mathcal{H}_1\in \mathcal{H}_1$, select all $\Nrv_2\mathcal{H}_1\in \mathcal{H}_1$ such that $\Nrv_1\mathcal{H}_1\ \sn\ \Nrv_2\mathcal{H}_1$, {\em i.e.}, the pair of homology nerves $\Nrv_1\mathcal{H}_1\in \mathcal{H}_1$ overlap (have strong proximity).   From Lemma~\ref{thm:sn-implies-near},  $\Nrv_1\mathcal{H}_1\ \dcap\ \Nrv_2\mathcal{H}_1\neq \emptyset$.   Hence, $\Nrv_1\mathcal{H}_1\ \dcap\ \Nrv_2\mathcal{H}_1\in \Phi(\mathcal{H}_1)$.\\
$\dcup$:  By definition,
\begin{align*}
\Nrv_1\mathcal{H}_1\ \dcup\ \Nrv_2\mathcal{H}_1 &= \{\cyc A\in \mathcal{H}_1: \cyc A \in \Nrv_1\mathcal{H}_1\ \dcap\ \Nrv_2\mathcal{H}_1\\ 
          &\mbox{or}\ \Phi(\cyc A) \in \Phi(\Nrv_1\mathcal{H}_1)\ \mbox{or}\ \Phi(\cyc A) \in \Phi(\Nrv_2\mathcal{H}_1)\}.
\end{align*}
Hence, $\Nrv_1\mathcal{H}_1\ \dcup\ \Nrv_2\mathcal{H}_1\in \Phi(\mathcal{H}_1)$.
\end{proof}

\begin{remark}
Lemma~\ref{lemma:LeaderUniformTopologyOnH1} is a stronger result than we need to derive a descriptive CW, which is a convenient setting for the study of finite, bounded planar shapes signatures.
Theorem~\ref{thm: LeaderUniformTopology} is a direct result of Lemma~\ref{lemma:LeaderUniformTopologyOnH1}.
\qquad \textcolor{blue}{\Squaresteel}
\end{remark}

\begin{theorem}\label{EHnerve}{\rm ~\cite[\S III.2, p. 59]{Edelsbrunner1999}}
Let $\mathscr{F}$ be a finite collection of closed, convex sets in Euclidean space.  Then the nerve of $\mathscr{F}$ and the union of the sets in $\mathscr{F}$ have the same homotopy type.
\end{theorem}

\begin{lemma}\label{lemma:homotopyType}
Let $\mathcal{H}_1$ be a collection of homology groups on a simplicial complex covering a finite, bounded shape.    Then a homology nerve $\Nrv \mathcal{H}_1\in 2^{\mathcal{H}_1}$ and $\mathop{\bigcup}\limits_{\cyc A\in \mathcal{H}_1}\cyc A$ have same homotopy type.
\end{lemma}
\begin{proof}$\mbox{}$\\
\noindent $\mathcal{H}_1$ is a collection of 1-cycles, which are closed, convex sets in Euclidean space.   
Then from Theorem~\ref{EHnerve}, $\Nrv \mathcal{H}_1$ and  $\mathop{\bigcup}\limits_{\cyc A\in \mathcal{H}_1}\cyc A$ have same homotopy type.
\end{proof}

\begin{theorem}
Let $\left(\mathcal{H}_1,\left\{\sn,\dnear\right\}\right)$ be a collection of homology groups $H_1$ equipped with a proximal relator on a simplicial complex covering a finite, bounded shape.    Then $\Phi(\Nrv \mathcal{H}_1))\in 2^{\mathbb{R}^n}$ (a description of a homology nerve) and $\mathop{\bigcup}\limits_{\Phi(\cyc A)\in \Phi(\Nrv \mathcal{H}_1)}\Phi(\cyc A)$ (union of the descriptions) have same homotopy type.
\end{theorem}
\begin{proof}$\mbox{}$\\
\noindent 
Each member of $\Phi(\mathcal{H}_1)$ is feature vector in $\mathbb{R}^n$ and each point in $\mathbb{R}^n$ is a closed, convex singleton set.   Then from Lemma~\ref{lemma:homotopyType}, $\Phi(\Nrv \mathcal{H}_1)$ and  $\mathop{\bigcup}\limits_{\Phi(\cyc A)\in \Phi(\Nrv \mathcal{H}_1)}\Phi(\cyc A)$ have same homotopy type.
\end{proof}

\begin{remark}{\bf Open Problems}.\\
Let $\sh A$ be a finite, bounded planar shape covered with a simplicial complex $K$ and let $H_1(K)$ be a homology group on $K$.\\
An open problem in shape theory is selecting each 1-cycle that is the contour of a subshape containing a hole in $\sh A$.

A second open problem in shape theory is the construction of a collection of homology nerves that overlap a subshape of interest in a shape $sh A$.

Let  $\mathcal{H}_1(K)$ be a collection of homology groups on a simplicial complex $K$.   A third open problem in shape theory is detecting space curves (also called twisted curves by D. Hilbert and S. Cohn-Vossen~\cite[\S 27]{Hilbert1932ChelseaSpaceCurves}) overlapping with 1-cycles in $\mathcal{H}_1(K)$.  

A fourth open problem in shape theory is to use homology nerves as a basis for measuring the persistence over time of object shapes in digital images.  

A fifth open problem in shape theory is to measure the persistence of a finite, bounded shape over time using a shape signature that includes the uniform iso-curvature of the 1-cycles and the Betti number of a homology group on the shape.  
\qquad \textcolor{blue}{\Squaresteel}
\end{remark}
   
  
\bibliographystyle{amsplain}
\bibliography{NSrefs}

\providecommand{\bysame}{\leavevmode\hbox to3em{\hrulefill}\thinspace}
\providecommand{\MR}{\relax\ifhmode\unskip\space\fi MR }
\providecommand{\MRhref}[2]{%
  \href{http://www.ams.org/mathscinet-getitem?mr=#1}{#2}
}
\providecommand{\href}[2]{#2}
\begin{thebibliography}{10}

\bibitem{AhmadPeters2017arXiv1706-04549v1spokeComplexes}
M.Z. Ahmad and J.F. Peters, \emph{Delta complexes in digital images.
  {A}pproximating image object shapes}, arXiv \textbf{1706} (2017),
  no.~04549v1, 1--20.

\bibitem{Alexandroff1926MAnnNerfTheorem}
P.~Alexandroff, \emph{Simpliziale approximationen in der allgemeinen
  topologie}, Mathematische Annalen \textbf{101} (1926), no.~1, 452--456,
  MR1512546.

\bibitem{Alexandroff1932elementaryConcepts}
\bysame, \emph{Elementary concepts of topology}, Dover Publications, Inc., New
  York, 1965, 63 pp., translation of Einfachste Grundbegriffe der Topologie
  [Springer, Berlin, 1932], translated by Alan E. Farley , Preface by D.
  Hilbert, MR0149463.

\bibitem{Alexandroff1935}
P.~Alexandroff and H.~Hopf, \emph{Topologie {I}. {B}erichtigter reprint. {D}ie
  grundlehren der mathematischen wissenschaften, {B}and 45}, Springer-Verlag,
  Berlin, 1935, xiii+636+2 pp., Reprint 1974, MR0345087.

\bibitem{Beer1993bk}
G.~Beer, \emph{Topologies on closed and closed convex sets. mathematics and its
  applications, 268.}, Kluwer Academic Publishers Group, Dordrecht, The
  Netherlands, 1993, xii+340, pp. ISBN: 0-7923-2531-1, MR1269778.

\bibitem{Munch2017TDArank}
A.~Benkhlifa and F.~Ghorbel, \emph{A 2{D} contour description generalized
  curvature scale space}, Time Delay Systems. Theory, Numerics, Applications
  and Experiments (T.~Insperger, T.~Ersal, and G.~Orosz, eds.), Springer Int.
  Pub. AG, 2017, pp.~93--106.

\bibitem{BenkhlifaGhorbel2017isoCurvature}
\bysame, \emph{A 2{D} contour description generalized curvature scale space},
  Representations, Analysis and Recognition of Shape and Motion from Imaging
  Data (B.B. Amor, F.~Chaieb, and F.~Ghirbel, eds.), Springer, 2017,
  pp.~129--140.

\bibitem{BergerGostiaux1988orientation}
M.~Berger and G.~Gostiaux, \emph{Differential geometry: manifolds, curves, and
  surfaces, graduate texts in mathematics, 115}, Springer-Verlag, New York,
  1988, {x}+474 pp. ISBN: 0-387-96626-9; Translated from the French by Silvio
  Levy; MR0917479.

\bibitem{Borsuk1948DecompositionsHomotopyType}
K.~Borsuk, \emph{On the imbedding of systems of compacta in simplicial
  complexes}, Fundamenta Mathematicae \textbf{35} (1948), 217--234, MR0028019.

\bibitem{Bredon1997homologyTheory}
G.E. Bredon, \emph{Topology and geometry}, Springer-Verlag, New York, 1997,
  xiv+557 pp. ISBN: 0-387-97926-3, MR1700700.

\bibitem{DiConcilio2016descriptiveProximity}
A.~Di Concilio, C.~Guadagni, J.F. Peters, and S.~Ramanna, \emph{Descriptive
  proximities {I}: {P}roperties and interplay between classical proximities and
  overlap}, arXiv \textbf{1609} (2016), no.~06246v1, 1--12, Math. in Comp. Sci.
  2017, \url{https://doi.org/10.1007/s11786-017-0328-y}, in press.

\bibitem{Dao2017arXivNerveComplexes}
H.~Dao, J.~Doolittle, K.~Duna, B.~Goeckner, B.~Holmes, and J.~Lyle,
  \emph{Higher nerves of simplicial complexes}, arXiv \textbf{1710} (2017),
  no.~06129v1, 1--18.

\bibitem{Concilio2009}
A.~{Di Concilio}, \emph{Proximity: {A} powerful tool in extension theory,
  functions spaces, hyperspaces, boolean algebras and point-free geometry},
  Beyond Topology, AMS Contemporary Mathematics 486 (F.~Mynard and E.~Pearl,
  eds.), Amer. Math. Soc., 2009, pp.~89--114.

\bibitem{Edelsbrunner1999}
H.~Edelsbrunner and J.L. Harer, \emph{Computational topology. {A}n
  introduction}, Amer. Math. Soc., Providence, RI, 2010, xii+241 pp. ISBN:
  978-0-8218-4925-5, MR2572029.

\bibitem{Efremovic1952}
V.A. Efremovi\v{c}, \emph{The geometry of proximity {I} (in {R}ussian)}, Mat.
  Sb. (N.S.) \textbf{31(73)} (1952), no.~1, 189--200.

\bibitem{Ghrist2008BAMSbarcodePersistence}
R.~Ghrist, \emph{Barcodes: the persistent topology of data}, Bull. Amer. Math.
  Soc. (N.S.) \textbf{45} (2008), no.~1, 61--75, MR2358377.

\bibitem{Grunbaum1970EMnerveComplexes}
B.~Gr\"{u}nbaum, \emph{Nerves of simplicial complexes}, Aequationes
  Mathematicae \textbf{4} (1970), no.~1-2, 63--73, MR0264648.

\bibitem{Guadagni2015thesis}
C.~Guadagni, \emph{Bornological convergences on local proximity spaces and
  $\omega_{\mu}$-metric spaces}, Ph.D. thesis, Universit\`{a} degli Studi di
  Salerno, Salerno, Italy, 2015, Supervisor: A. {Di Concilio}, 79pp.

\bibitem{Hatcher2002CUPalgebraicTopology}
A.~Hatcher, \emph{Algebraic topology}, Cambridge University Press, Cambridge,
  UK, 2002, {x}ii+544 pp. ISBN: 0-521-79160-X, MR1867354.

\bibitem{Hilbert1932ChelseaSpaceCurves}
D.~Hilbert and S.~Cohn-Vossen, \emph{Geometry and the imagination [naglyadnaya
  geometriya (russian)}, Chelsea Publishing Company, New York, 1952, ix+357
  pp., translation by P. Neményi, MR0046650.

\bibitem{Janich2005closedWeakTopology}
K.~J\"{a}nich, \emph{Topologie. ({G}erman) [{T}opology], 8th ed.},
  Springer-Verlag, Berlin, 2005, {x}+239 pp. ISBN: 978-3-540-21393-2,
  MR2262391.

\bibitem{Kuratowsk1933FMfundamentalNerveTheorem}
C.~Kuratowski, \emph{Sur un th\'{e}or\`{e}me fondamental concernant le nerf d'
  un syst\`{e}me d' ensembles}, Fundamenta Mathematicae \textbf{20} (1933),
  no.~1, 191--196, zbMATH JFM 59.0561.01.

\bibitem{Leader1959}
S.~Leader, \emph{On clusters in proximity spaces}, Fundamenta Mathematicae
  \textbf{47} (1959), 205--213.

\bibitem{Lodato1962}
M.W. Lodato, \emph{On topologically induced generalized proximity relations,
  {P}h.{D}. thesis}, Rutgers University, 1962, supervisor: S. Leader.

\bibitem{Naimpally2013}
S.A. Naimpally and J.F. Peters, \emph{Topology with applications. topological
  spaces via near and far}, World Scientific, Singapore, 2013, xv + 277 pp,
  Amer. Math. Soc. MR3075111.

\bibitem{Naimpally70withWarrack}
S.A. Naimpally and B.D. Warrack, \emph{Proximity spaces}, Cambridge Tract in
  Mathematics No. 59, Cambridge University Press, Cambridge, UK, 1970, x+128
  pp.,Paperback (2008), MR2573941.

\bibitem{PersoonFu1986PAMIFourierDescriptors}
E.~Persoon and K.-S. Fu, \emph{Shape discrimination using {F}ourier
  descriptors}, J. London Math. Soc. \textbf{PAMI-8} (1986), no.~3, 388--397.

\bibitem{PetersInan2016spNerves}
J.~F. Peters and E.~Inan, \emph{Strongly proximal {E}delsbrunner-{H}arer
  nerves}, Proc. Jangjeon Math. Soc. \textbf{19} (2016), no.~3, 1--20,
  MR3618825, zbMATH Zbl 1360.54021.

\bibitem{Peters2013mcs}
J.F. Peters, \emph{Local near sets: {P}attern discovery in proximity spaces},
  Math. in Comp. Sci. \textbf{7} (2013), no.~1, 87--106, DOI
  10.1007/s11786-013-0143-z, MR3043920, ZBL06156991.

\bibitem{Peters2013mcsintro}
\bysame, \emph{Near sets: {A}n introduction}, Math. in Comp. Sci. \textbf{7}
  (2013), no.~1, 3--9, DOI 10.1007/s11786-013-0149-6, MR3043914.

\bibitem{Peters2014book}
\bysame, \emph{Topology of digital images. {V}isual pattern discovery in
  proximity spaces}, Intelligent Systems Reference Library, vol.~63, Springer,
  2014, xv + 411pp, Zentralblatt MATH Zbl 1295 68010.

\bibitem{Peters2015JangjeonMSstrongProximity}
\bysame, \emph{Proximal {D}elaunay triangulation regions}, Proceedings of the
  Jangjeon Math. Soc. \textbf{18} (2015), no.~4, 501--515, MR3444736.

\bibitem{Peters2016ComputationalProximity}
\bysame, \emph{Computational proximity. {E}xcursions in the topology of digital
  images.}, Intelligent Systems Reference Library \textbf{102} (2016), xxviii +
  433pp, DOI: 10.1007/978-3-319-30262-1.

\bibitem{Peters2016relator}
\bysame, \emph{Proximal relator spaces}, Filomat \textbf{30} (2016), no.~2,
  469--472, doi:10.2298/FIL1602469P, MR3497927.

\bibitem{Peters2017arXiv1704-05909spokes}
\bysame, \emph{Proximal nerve complexes. {A} computational topology approach},
  Set-Value Mathematics and Applications \textbf{1} (2017), no.~1, 1--16, arXiv
  preprint arXiv:1704.05909.

\bibitem{Peters2017arXiv1708-04147planarShapes}
\bysame, \emph{Proximal planar shapes. {C}orrespondence between shape and nerve
  complexes}, arXiv \textbf{1708} (2017), no.~04147v1, 1--12.

\bibitem{Peters2015AMSJmanifolds}
J.F. Peters and C.~Guadagni, \emph{Strong proximities on smooth manifolds and
  {V}orono\"{i} diagrams}, Advances in Math.: Sci. J. \textbf{4} (2015), no.~2,
  91--107, Zbl 1339.54020.

\bibitem{PetersGuadagni2015stronglyNear}
\bysame, \emph{Strongly near proximity and hyperspace topology}, arXiv
  \textbf{1502} (2015), no.~05913, 1--6.

\bibitem{Szaz1987}
\'{A} Sz\'{a}z, \emph{Basic tools and mild continuities in relator spaces},
  Acta Math. Hungar. \textbf{50} (1987), no.~3-4, 177--201, MR0918156.

\bibitem{Cech1966}
E.~\u{C}ech, \emph{Topological spaces}, John Wiley \& Sons Ltd., London, 1966,
  fr seminar, Brno, 1936-1939; rev. ed. Z. Frolik, M. Kat\u{e}tov. Scientific
  editor, Vlastimil Pt '{a}k. Editor of the English translation, Charles O.
  Junge Publishing House of the Czechoslovak Academy of Sciences, Prague;
  Interscience Publishers John Wiley \& Sons, London-New York-Sydney 1966 893
  pp., MR0211373.

\bibitem{Ulrich1970SIAMJAMorientedGraph}
J.W. Ulrich, \emph{A characterization of planar oriented graphs}, SIAM J. Appl.
  Math. \textbf{18} (1970), MR0255445.

\bibitem{Vick1994homologyTheory}
J.W. Vick, \emph{Homology theory. {A}n introduction to algebraic topology. 2nd
  ed., graduate texts in mathematics, 145}, Springer-Verlag, New York, 1994,
  {x}iv+242 pp. ISBN: 0-387-94126-6, MR1254439.

\bibitem{Wallman1938}
H.~Wallman, \emph{Lattices and topological spaces}, Annals of Math. \textbf{39}
  (1938), no.~1, 112--126, MR1503392, {\em cf.} H. Wallman, Lattices and
  topological spaces, Thesis (Ph.D.), Princeton University. 1937.

\bibitem{Whitehead1949BAMS-CWtopology}
J.H.C. Whitehead, \emph{Combinatorial homotopy. {I}.}, Bulletin of the American
  Mathematical Society \textbf{55} (1949), no.~3, 213--245, Part 1, MR0030759.

\end{thebibliography}

\end{document}